\documentclass[a4paper,10pt]{amsart}
\usepackage{ae,amsfonts,euscript,enumerate}
\usepackage{amsmath,amssymb,amsthm}
\usepackage{enumerate}
\usepackage{microtype}
\usepackage{colortbl}
\usepackage{array}
\usepackage[latin1]{inputenc}

\newcommand{\mathbbm}[1]{\text{\usefont{U}{bbm}{m}{n}#1}}
\DeclareMathOperator{\Cayley}{Cay}

\newcommand{\ZZ}{\ensuremath{\mathbb Z}}

\newtheorem{thm}{Theorem}[section]
\newtheorem{lem}[thm]{Lemma}
\newtheorem{prop}[thm]{Proposition}
\newtheorem{cor}[thm]{Corollary}

\theoremstyle{definition}

\newtheorem{rem}[thm]{Remark}
\newtheorem{quest}{Question}

\title{Integral Cayley Graphs and Groups}

\date{\today}

\author{Azhvan Ahmady}
 \address{Department of Mathematics, Simon Fraser University, Burnaby, B.C. V5A 1S6}
 \email{ssheikha@sfu.ca}
\author{Jason P.~Bell}
 \thanks{The authors thank NSERC for its generous support.}
 \address{Department of Pure Mathematics, University of Waterloo, Waterloo, Canada}
 \email{jpbell@uwaterloo.ca}
\author{Bojan Mohar}
 \address{Department of Mathematics, Simon Fraser University, Burnaby, B.C. V5A 1S6}
 \email{mohar@sfu.ca}

\subjclass[2010]{05C50, 20C10 (primary), 05C25, 05E10 (secondary)}%
\keywords{Cayley graphs, Cayley groups, CIS groups, Weakly Cayley integral groups, integral eigenvalues}%

\begin{document}
\bibliographystyle{plain}

\begin{abstract}
We solve two open problems regarding the classification of certain classes of Cayley graphs with integer eigenvalues.  We first classify all finite groups that have a ``non-trivial'' Cayley graph with integer eigenvalues, thus solving a problem proposed by Abdollahi and Jazaeri. The notion of Cayley integral groups was introduced by Klotz and Sander. These are groups for which every Cayley graph has only integer eigenvalues. In the second part of the paper, all Cayley integral groups are determined.
\end{abstract}
\maketitle

\section{Introduction}

A graph $X$ is said to be \emph{integral\/} if all eigenvalues of the adjacency matrix
of $X$ are integers.  This property was first defined by Harary and Schwenk \cite{HS} who
suggested the problem of classifying integral graphs. This problem ignited a significant investigation among algebraic graph theorists, trying to construct and classify integral graphs. Although this problem is easy to state, it turns out to be extremely hard. It has been attacked by many mathematicians during the last forty years and it is still wide open. 

Since the general problem of classifying integral graphs seems too difficult, graph theorists started to investigate special classes of graphs, including trees, graphs of bounded degree, regular graphs and Cayley graphs. What proves so interesting about this problem is that no one can yet identify what the integral trees are or which $5$-regular graphs are integral.

The notion of \emph{CIS groups}, that is, groups admitting no integral Cayley graphs besides complete multipartite graphs, was introduced by Abdollahi and Jazaeri \cite{AB}, who classified all abelian CIS groups. The question of which non-abelian groups are CIS remained open, however.
A similar but more intriguing notion of Cayley integral groups was introduced by Klotz and Sander in \cite{KS}.  These are finite groups $G$ with the property that whenever $S$ is a symmetric generating set for $G$, the Cayley graph of $G$ with respect to the generating set $S$ is an integral graph.  Klotz and Sander classified all such groups in the abelian group case, while the general case was left open.

The main results in this paper are Theorems \ref{thm: CIS} and \ref{thm: main} in which we respectively classify all CIS groups and all Cayley integral groups. Although the first of these two results is not very hard to prove, we find it interesting enough since it shows, in particular, that every finite non-abelian group admits a non-trivial Cayley graph whose eigenvalues are all integral.  The classification of Cayley integral groups is more difficult and requires new methods.  In this case we are able to give a complete list of all Cayley integral groups up to isomorphism, with the conclusion being that, aside from a few sporadic examples, all such groups are either abelian or isomorphic to the direct product of the quaternion group of order $8$ with an elementary  abelian $2$-group.

As the setting of integral Cayley graphs suggests, all groups in this paper will be finite.

\section{Group representations and eigenvalues of Cayley graphs}

We require some definitions and notation from representation theory, graph theory and group theory. We will use the standard notation, for a more detailed account the reader is referred to \cite{FH,GS, RO}.
All groups considered are finite (written multiplicatively), and all fields are subfields of the complex numbers.
For two matrices $A$ and $B$ we denote their \emph{Kronecker product} by $A\otimes B$. For an integer $k\ge1$, $I_k$ denotes  the $k\times k$ identity matrix and $J_k$ represents the $k\times k$ matrix of all ones.
For a symmetric (self-inverse) subset $S$ of $G$, we define the \emph{Cayley graph} of $G$ over $S$, denoted $\Cayley(G,S)$, to be the graph with vertex set $G$ and $x,y \in G$ adjacent if $xy^{-1} \in S$.
A group $G$ is called \emph{perfect} if it is equal to its derived subgroup, i.e. $G=G'$.
We denote the group algebra of $G$ over the field $\mathbb{F}$ by $\mathbb{F}G$. That is, $\mathbb{F}G$ is the vector space over $\mathbb{F}$ with basis $G$ and multiplication defined by extending the group multiplication linearly.
Identifying $\sum_{g \in G} a_g g$ with the function $g \mapsto a_g$, we can view the vector space $\mathbb{F}G$ as the space of all $\mathbb{F}$-valued functions on $G$.
We sometimes identify a subset $S$ of $G$ with the element $\sum\limits_{s\in S} s$ of the group algebra $\mathbb{C}G$.

Let $V$ be an $n$-dimensional vector space over $\mathbb{C}$. A \emph{complex representation} (or simply a \emph{representation}) of $G$ on $V$ is a group homomorphism $\rho: G \rightarrow GL(V)$, where $GL(V)$ denotes the group of invertible endomorphism of $V$. The \emph{degree} of $\rho$ is the dimension of $V$.
Two representations $\rho_1$ and $\rho_2$ of $G$ on $V_1$ and $V_2$, respectively, are \emph{equivalent} (written $\rho_1 \cong \rho_2$) if there is a linear isomorphism $T: V_1 \to V_2$ such that for every $g\in G$ we have $T\rho_1(g) = \rho_2(g)T$.

If $\rho$ is a representation of $G$ then the \emph{character} $\chi_{\rho}$ \emph{afforded} by $\rho$, is the linear functional $\chi_{\rho}: \mathbb{C}G\rightarrow \mathbb{C}$ defined by
$$\chi_{\rho}(g)= {\rm tr}(\rho(g)),\quad g \in G,$$
and extended by linearity to $\mathbb{C}G$ (the trace ${\rm tr} (\alpha)$ of a linear map $\alpha$ is the trace of any matrix representing $\alpha$ according to some basis).
The \emph{degree} of the character $\chi_{\rho}$ is the degree of  $\rho$, and is equal to $\chi_{\rho}(1)$. A character of degree one is called a \emph{linear character}. The index $[G:G']$ of derived subgroup $G'$, is equal to the number of
linear characters of $G$. The character $\chi$ which assigns $1$ to every element of a group $G$ is called the \emph{principal character} of $G$ and is denoted by ${\mathbbm 1}_G$.

The $|G|$-dimensional representation $\rho_{{\rm reg}}:G \rightarrow GL(\mathbb{C}G)$ defined by $\rho_{{\rm reg}}(g)(x)=gx$ ($g \in G, x \in \mathbb{C}G)$) is called the
\emph{left-regular representation}. Choosing $G$ as a basis for $\mathbb{C}G$, we see that for every $g \in G$, $[\rho_{{\rm reg}}(g)]_G$ (the matrix representing the linear map $\rho_{{\rm reg}}(g)$ according to the basis $G$) is the $(|G| \times |G|)$-matrix, indexed with the elements of $G$, such that for every $k,h \in G$

\[([\rho_{{\rm reg}}(g)]_G)_{h,k} = \left\{
\begin{array}{l l}
  1 & \quad \mbox{if $h=gk$ }\\
  0 & \quad \mbox{otherwise.}
\end{array} \right. \]

This gives a natural link to Cayley graphs since the adjacency matrix of a Cayley graph $\Cayley(G,S)$ can be written as
\begin{equation}
    A = \sum_{s\in S}\ [\rho_{{\rm reg}}(s)]_G.
\label{eq:adjacency regular representations}
\end{equation}

Let $\rho: G \rightarrow GL(V)$ be a representation. A subspace $W$ of $V$ is said to be $\rho$\emph{-invariant}, if $\rho(g)w \in W$ for every $g \in G$ and $w \in W$.  If $W$ is a $\rho$-invariant subspace of $V$, then the restriction of $\rho$ to $W$, that is $\rho|_W:G \rightarrow GL(W)$, is a representation of $G$ on $W$. If $V$ has no non-trivial $\rho$-invariant subspaces, then $\rho$ is said to be an \emph{irreducible representation} for $G$ and the corresponding character $\chi_{\rho}$ an \emph{irreducible character} for $G$.

If $G$ is abelian, then every irreducible representation $\rho$ of $G$ is 1-dimensional and thus it can be identified with its character $\chi_\rho$.

A result of Diaconis and Shahshahani \cite{DS} shows a link between the representation theory and spectral theory of Cayley graphs.
Let\/ $\rho_1,\dots,\rho_k$ be a complete set of irreducible representations of a group $G$, and suppose that $S$ is a symmetric subset of $G$. For\/ $t=1,\dots,k$, let $d_t$ be the degree of $\rho_t$, and let $\Lambda_t$ be the (multi-)set of eigenvalues of the matrix $\rho_t(S)=\sum\limits_{s \in S}\rho_t(s)$. Then the following holds:
\begin{enumerate}
\item[\rm (1)] The (multi-)set of eigenvalues of\/ $\Cayley(G,S)$ equals
$\cup_{t=1}^k \Lambda_t$.
\item[\rm (2)] If the eigenvalue $\lambda$ occurs with multiplicity $m_t(\lambda)$ in
$\rho_t(S)$ $(1\le t\le k)$, then the multiplicity of $\lambda$ in $\Cayley(G,S)$ is $\sum_{t=1}^k d_t m_t(\lambda)$.
\end{enumerate}

If $G$ is a group and $\chi$ a character of $G$, then the character values are all algebraic integers and thus they are rational if and only if they are integral.
Observe that every representation of $G$ (and in particular the left-regular representation) is a direct sum of some copies of irreducible representations. This fact and the aforementioned result of Diaconis and Shahshahani yield the following criterion for integrality of Cayley graphs that will be used throughout the paper.

\begin{prop}
\label{prop:integrality criterion}
A connected Cayley graph $\Cayley(G,S)$ is integral if and only if for every representation $\rho$ of $G$, $\rho(S)$ is integral.
\end{prop}

\section{Cayley integral simple groups}

In this section, we answer a question of Abdollahi and Jazaeri \cite{AB} concerning Cayley integral simple groups. Abdollahi and Jazaeri defined a \emph{Cayley integral simple} group (CIS group for short) to be a group $G$ with the property that the only connected integral Cayley graphs of $G$ are complete multipartite graphs.  In addition to this, they noticed that given a symmetric generating subset $S$ of $G$, $\Cayley(G,S)$ is a complete multipartite graph if and only if $S$ is the complement of a subgroup of $G$.  Thus a simpler definition of a CIS group is that it is a group $G$ with the property that for a symmetric generating set $S$ of $G$, we have that $\Cayley(G,S)$ is an integral graph if and only if $S$ is the complement of a subgroup of $G$.

As part of their study of CIS groups, Abdollahi and Jazaeri gave a complete characterization of abelian CIS groups, which we state now.
\begin{thm}[Abdollahi and Jazaeri \cite{AB}]
\label{AB}
Let $G$ be an abelian group. Then $G$ is a CIS group if and only if $G \cong \ZZ_{p^2}, \ZZ_p$ for some prime number $p$, or $G\cong \ZZ_2 \times \ZZ_2$.
\end{thm}

In addition to this, they posed the following question.
\begin{quest}
Which finite non-abelian groups are CIS groups?
\end{quest}

We answer their question with the perhaps surprising answer that non-abelian finite CIS groups do not exist. More formally, we show the following result.

\begin{thm}
Let $G$ be a CIS group. Then $G$ is abelian and in particular is isomorphic to either a cyclic group of order $p$ or $p^2$ for some prime $p$, or is isomorphic to $\mathbb{Z}_2^2$.
\label{thm: CIS}
\end{thm}

Our proof of Theorem \ref{thm: CIS} is essentially an induction argument on the order of the group and to make the argument work we need the following result.

\begin{thm}[Miller and Moreno \cite{MM}]
Let $G$ be a non-abelian group with the property that every proper subgroup is abelian.  Then $|G|$ has at most two prime divisors and there is some prime $p$ dividing the order of $G$ such that the Sylow $p$-subgroup of $G$ is normal.
\label{thm: MM}
\end{thm}

We will also need a basic result showing that the CIS property is closed under the process of taking subgroups and homomorphic images.

\begin{lem}
\label{subgroup CIS}
Let $G$ be a CIS group. Then every subgroup and every homomorphic image of $G$ is also a CIS group.
\end{lem}

\begin{proof}
Let us assume $H$ is a subgroup of $G$. We will show that $H$ is a CIS group. Suppose, towards a contradiction, that $H$ is not a CIS group.
Then there is a self-inverse generating subset $S$ of $H$ such that $1\not\in S$, $H\setminus S$ is not a subgroup of $H$, and such that $\Cayley(H,S)$ is an  integral graph.
We take $T=S\cup (G\setminus H)$.  Then $T$ generates $G$, $1\not\in T$, and $G\setminus T=H\setminus S$ is not a subgroup of $G$.  Then the adjacency matrix of $\Cayley(G,T)$ is given by
$B:=A_S\otimes I_k + J_n \otimes (J_k-I_k)$, where $A_S$ is the adjacency matrix of $\Cayley(H,S)$, $n=|H|$, $k=[G:H]$.  Since $A_S$ is Hermitian, it is unitarily diagonalizable.  In particular, we can find an orthogonal basis of eigenvectors of $A_S$, $\{{\bf w}_1,\ldots ,{\bf w}_n\}$.  We may assume that ${\bf w}_1={\bf j}$, the vector whose coordinates are all equal to one.  Thus there are integers $|S|=\lambda_1,\ldots ,\lambda_n$ such that $A_S{\bf w}_i=\lambda_i{\bf w}_i$ for $i=1,\ldots ,n$.  Notice that since $J_n$ has each of its columns equal to ${\bf j}={\bf w}_1$, we have $J_n{\bf w}_1 = n{\bf w}_1$ and
$J_n{\bf w}_i = 0$ for $i\ge2$.  Since $J_k$ is symmetric and of rank $1$, it has eigenvectors
${\bf u}_1,\ldots ,{\bf u}_k$ with $J_k{\bf u}_i =k\delta_{1,i} {\bf u}_i$.  Then ${\mathcal S} = \{{\bf w}_i\otimes {\bf u}_j \mid 1\le i\le n, \, 1\le j\le k\}$ is a basis of $\mathbb{C}^n\otimes \mathbb{C}^k$.  Observe that
$B({\bf w}_i\otimes {\bf u}_j) = (\lambda_i+\delta_{1,i}n(k\delta_{1,j}-1)){\bf w}_i\otimes {\bf u}_j$.  In particular, ${\mathcal S}$ is a complete set of eigenvectors of $B$ and thus every eigenvalue of $B$ is an integer.  Thus $G$ is not a CIS group, a contradiction.

For the second part, let us assume $N$ is a proper normal subgroup of $G$. We show, using a proof by contradiction, that $G/ N$ is a CIS group.
If $G/N$ is not a CIS group, then there exists a self-inverse generating subset $\bar{S}$ of $G/N$ such that $N\not\in \bar{S}$, $G/N\setminus \bar{S}$ is not a subgroup of $G/N$ and $\Cayley(G/N,\bar{S})$ is an  integral graph.
Let us assume $\bar{S}=\{Ns \mid s \in S\}$, where the set $S$ is symmetric in $G$. We have $S\cap N=\emptyset$. If we define $T=\mathop{\cup}_{s\in S}Ns$, then clearly $T$ is a self-inverse subset of $G$ such that $1\not\in T$.
Since $G/N\setminus \bar{S}$ is not a subgroup of $G/N$, there are $g_1$ and $g_2$ in $G$ such that $Ng_1$ and $Ng_2$ do not belong to $\bar{S}$ but $Ng_1g_2\in \bar{S}$. Thus, we get $g_1g_2\in T$ and $\{g_1, g_2\}\cap T=\emptyset$.
This proves that $G \setminus T$ is not a subgroup of $G$. It is easy to see that the adjacency matrix of $\Cayley(G,T)$ is $J_k\otimes A_{G/N}$, where $k =|N|$ and $A_{G/N}$ is the adjacency matrix of $\Cayley(G/N,\bar{S})$.
Since $J_k$ has eigenvalues $0$ with multiplicity $k-1$  and $k$ with multiplicity $1$, we deduce that $\Cayley(G,T)$ is an integral graph and thus $G$ is not a CIS group which contradicts our assumption.
Therefore, $G/N$ is a CIS group.
\end{proof}

For the sake of completeness, we present a rather short proof of Theorem \ref{AB} based on above result.

\begin{proof}[Proof of Theorem \ref{AB}.]
It is easy to see that $\ZZ_2\times \ZZ_2$, $\ZZ_{p^2}$, and $\ZZ_p$ (where $p$ is a prime) are CIS groups.
To prove that there are no other abelian CIS groups, let us assume that $G$ is an abelian CIS group of minimum order that is not isomorphic to $\ZZ_2\times \ZZ_2$, $\ZZ_{p^2}$, $\ZZ_p$.

Suppose first that $G=A \times\ZZ_m$, where $m>2$ and $|A|\ge2$.
Let $S_1=A\setminus \{0\}$, $S_2=\ZZ_m\setminus \{0\}$ and $S=S_1\cup S_2$.
We notice that the Cayley graph $\Cayley(G,S)$ is isomorphic to the Cartesian product
$\Cayley(A,S_1) \,\Box\, \Cayley(\ZZ_m,S_2)$.
Since $\Cayley(A,S_1)$ and $\Cayley(\ZZ_m,S_2)$ are complete graphs and therefore integral, their Cartesian product $\Cayley(G,S)$ is also integral. Since $G\setminus S$ is not a subgroup of $G$, we conclude that $G$ is not a CIS group.

Since every abelian group is a direct product of cyclic subgroups whose orders are powers of primes, we conclude from the above that $G$ is either isomorphic to $\ZZ_{p^n}$ or to $\ZZ_2^n$ for some $n>2$. Since subgroup of CIS groups are CIS, and since $\ZZ_{p^n}$ ($\ZZ_2^n$) contains $\ZZ_{p^{n-1}}$ ($\ZZ_2^{n-1}$) as a subgroup, we may assume that $n=3$.
However, $\ZZ_{p^3}$ and $\ZZ_2^3$ are not CIS as evidenced by the following generating sets.
For $\ZZ_2^3$ we take $S=\{(0,0,1),(0,1,0),(1,0,0)\}$ whose Cayley graph is the 3-cube having only integral eigenvalues.
For $\ZZ_{p^3} = \langle a \mid a^{p^3}=1\rangle$, we let $Y = \{a^{tp}\mid t=0,1,\dots,p^2-1\}$ and $Z = \{a^{tp^2}\mid t=0,1,\dots,p-1\}$. Now, take $S = \ZZ_{p^3}\setminus(Y\setminus Z)$ whose complement contains $a^p$ but not $a^{p^2}$, and is therefore not a subgroup. The corresponding Cayley graph is easily seen to be integral. This completes the proof.
\end{proof}

Our inductive proof of Theorem \ref{thm: CIS} requires considering a few base cases, which are covered by the following lemma.

\begin{lem}
\label{bc}
Suppose that $G$ is one of the following groups:
$D_4$, $Q_8$, $A_4$, or a non-abelian semi-direct product of two cyclic groups of prime order.
Then $G$ is not a CIS group.
\end{lem}

\begin{proof}
We write $D_4=\langle x,y~|~x^4=y^2=1, yxy=x^{-1}\rangle$. Then $S=\{x, x^3,y\}$ is a symmetric generating set whose complement is not a subgroup of $D_4$. The Cayley graph $\Cayley(D_4,S)$ is isomorphic to the graph of the 3-cube which is integral. Thus $D_4$ is not a CIS group.


Let us now consider the group of quaternions, $Q_8=\{\pm i,\pm j, \pm k,\pm 1\}$.  We let $S=\{\pm i,\pm j,-1\}$.  Then $S$ is a symmetric generating set and its complement is not a subgroup.  If $\theta$ is the element of the group algebra of $Q_8$ corresponding to $S$ then $\theta$ is sent to $-I$ in the irreducible $2$-dimensional representation of $Q_8$; in the one-dimensional representations it is sent to an element in $\{-3, 1, 5\}$. This implies that $\Cayley(Q_8,S)$ is integral, and that
$Q_8$ is not a CIS group.

We next consider $A_4$.
We take $S=\{(12)(34), (123),(132),\allowbreak (124),(142),\allowbreak (234),\allowbreak (243),\allowbreak (134),(143)\}\subseteq A_4$.  The complement is not a subgroup since $(13)(24)$ and $(14)(23)$ are in the complement and their product is in $S$. Let $\theta$ denote the element
$$(12)(34)+(123)+(132)+(124)+(142)+(234)+(243)+(134)+(143)$$
of the group algebra of $A_4$ corresponding to $S$. Since the sum of all three-cycles, which we call $z$, is central in $A_4$, we have that $\rho(z)$ is a scalar multiple of the identity for every irreducible representation $\rho$ of $A_4$.  Since all characters of $A_4$ are integer-valued, we see that $\rho(z)$ must be an integral multiple of the identity.  Thus $\rho(\theta)$ has eigenvalues equal to the eigenvalues of $\rho((12)(34))$ shifted by an integer.  The eigenvalues of $\rho((12)(34))$ are in $\{\pm 1\}$, since $(12)(34)$ has order $2$.  This shows that all eigenvalues of $\rho(\theta)$ are integers.  Thus $A_4$ is not a CIS group.

Finally, let $G$ be a non-abelian semidirect product of two groups of prime order. Since every semidirect product of $\ZZ_p$ with $\ZZ_p$ and every semidirect product of $\ZZ_2$ with $\ZZ_p$ is abelian, we may assume that $G=\langle x~|~x^p=1\rangle \rtimes \langle y~|~y^q=1\rangle$, where $2 < p < q$ are distinct primes. Let $S=\{x,x^2,\ldots ,x^{p-1},y,y^2,\ldots ,y^{q-1}\}$. Then $S$ is symmetric and generates $G$. Since $p>2$, $|G\setminus S|=pq-p-q+2$ does not divide $pq$, thus $G\setminus S$ is not a group. Notice that the element $z=x+x^2+\cdots +x^{p-1}$ is central in the group algebra of $G$.  It follows that if $\phi$ is any irreducible representation of $G$ then $\phi(z)$ is a scalar multiple of the identity.  In fact, since $x$ has order $p$, all of the eigenvalues of $\phi(x)$ are $p$-th roots of unity and hence $\phi(z)=(p-1)I$ (if $\phi(x)=1$), or $\phi(z)=-I$ (when $\phi(x)\ne1$). Thus, $\phi(z)$ has eigenvalues in $\{-1,p-1\}$.Similarly, since the eigenvalues of $\phi(y)$ are $q$-th roots of unity, we conclude that $\phi(y)+\phi(y^2)+\cdots +\phi(y^{q-1})$ has eigenvalues in $\{-1,q-1\}$.  Thus $z+y+\cdots +y^{q-1}$ has eigenvalues in $\{-2,q-2,p-2,p+q-2\}$. This implies that $\Cayley(G,S)$ is an integral graph. Thus $G$ is not a CIS group.
\end{proof}

\begin{proof}[Proof of Theorem \ref{thm: CIS}.]
Suppose, towards a contradiction, that there exists a non-abelian CIS group. We pick such a CIS group $G$ of minimum order. By Lemma \ref{subgroup CIS}, every subgroup of $G$ is a CIS group. Thus, by minimality of $G$, we have that $G$ is a non-abelian group with the property that every proper subgroup of $G$ is abelian.

By Theorem \ref{thm: MM}, $G$ is either a $p$-group or there exist distinct primes $p$ and $q$ such that $|G|=p^aq^b$ for some positive integers $a$ and $b$ and the Sylow $p$-subgroup of $G$ is normal.  We consider these cases separately.

\medskip
\noindent \emph{Case I.} $G$ is a $p$-group.

\medskip

Suppose that $|G|=p^k$. Since $G$ is non-abelian, we have $k\ge3$. Let us first assume that $k=3$.
Every $p$-group has a non-trivial center. Since $G$ is not abelian, $G/Z(G)$ can not be cyclic. This implies that $|Z(G)|=p$ and $G/Z(G)\cong \ZZ_p\times\ZZ_p$.
If $G$ is a CIS group, then $\ZZ_p\times\ZZ_p$ should be a CIS group as well (Lemma \ref{subgroup CIS}). Thus by Theorem \ref{AB} we have $p=2$. Thus $G \cong Q_8$ or $D_4$, but according to Lemma \ref{bc}, $Q_8$ and $D_4$ are not CIS groups.
Thus no $p$-group of order $p^3$ is a CIS group. If $G$ is a non-abelian $p$-group of order greater than $p^3$ then $G$ has a subgroup of order $p^3$ and thus it is not a CIS group.


\medskip
\noindent \emph{Case II.} $G$ has order $p^aq^b$, where $p$ and $q$ are distinct primes, $a,b\ge 1$, and $G$ has a normal Sylow $p$-subgroup.

\medskip

In this case the Sylow $p$- and Sylow $q$-subgroups are proper and hence must be abelian by the minimality assumption on the order of $G$.  Let $P$ and $Q$ denote the Sylow $p$- and the Sylow $q$-subgroup of $G$, respectively.  Then $G\cong P\rtimes Q$.  We consider the case that $p=2$ and $p\neq 2$ separately.

\medskip
\noindent \emph{Subcase I:} $p$ is odd.

\medskip
\noindent Since $P$ is an abelian CIS group, we have $P\cong \mathbb{Z}_p$ or $P\cong \mathbb{Z}_{p^2}$.  Since $\mathbb{Z}_{p^2}$ has a characteristic subgroup of size $p$ we see that $G$ has a subgroup that is isomorphic to $\mathbb{Z}_p\rtimes Q$.  Also, $Q$ has a subgroup isomorphic to $\mathbb{Z}_q$ and since this normalizes the copy of $\mathbb{Z}_p$ we see that $G$ has a subgroup isomorphic to $\mathbb{Z}_p\rtimes \mathbb{Z}_q$.  Since there are no abelian CIS groups of order $pq$ we see that $\mathbb{Z}_p\rtimes \mathbb{Z}_q$ is non-abelian and so by minimality of $G$ we have $G\cong \mathbb{Z}_p\rtimes \mathbb{Z}_q$.  Since $G$ is non-abelian we have the result from Lemma \ref{bc}.

\medskip
\noindent \emph{Subcase II:} $p=2$.

\medskip

\noindent In this case, $|P|\in \{2,4\}$.  If $P$ is cyclic then $P$ has a characteristic subgroup isomorphic to $\mathbb{Z}_2$ and thus $G$ contains a copy of $\mathbb{Z}_2\rtimes Q$.  Notice that $\mathbb{Z}_2$ has only the trivial automorphism and so $\mathbb{Z}_2\rtimes Q\cong \mathbb{Z}_2\times Q$, which is not a CIS group since all abelian CIS groups have order a power of a prime.  Thus $P\cong \mathbb{Z}_2\times \mathbb{Z}_2$.  Notice that ${\rm Aut}(P)\cong\mathbb{Z}_3$ and so $P\rtimes Q$ is abelian unless $q=3$.  Since $G$ is non-abelian, we conclude that $q=3$ and that
$Q\cong \mathbb{Z}_3$ or $\mathbb{Z}_{9}$.  Notice that in either case, $Q$ has a subgroup of size $3$ that normalizes $P$ and so $G$ contains a subgroup isomorphic to $(\mathbb{Z}_2)^2\rtimes \mathbb{Z}_3$.  Since this group is necessarily a CIS group and since there are no abelian CIS groups of order $12$, we see that $(\mathbb{Z}_2)^2\rtimes \mathbb{Z}_3$ is a non-abelian semi-direct product and hence isomorphic to $A_4$.   But $A_4$ is not a CIS group by Lemma \ref{bc} (iii).
Thus we see that we cannot have $p=2$.

We have obtained a contradiction in each case and so we conclude that every CIS group is abelian.
\end{proof}


\section{Cayley Integral Groups}

Klotz and Sander \cite{KS} introduced the notion of a \emph{Cayley integral group}.  This is a group $G$ with the property that for every symmetric subset $S$ of $G$, $\Cayley(G,S)$ is an integral graph.  One of their results was a characterization of abelian integral groups.

\begin{thm}[Klotz and Sander \cite{KS}]
The only abelian Cayley integral groups are $$\ZZ_2^n \times \ZZ_3^m,~{\rm and}~\ZZ_2^n \times \ZZ_4^m,$$
where $m$ and $n$ are arbitrary non-negative integers.
\label{thm: KS}
\end{thm}

The main result of this paper is a complete characterization of Cayley integral groups, which we now state.

\begin{thm}
\label{thm: main}
The only Cayley integral groups are
$$\ZZ_2^n \times \ZZ_3^m,\ \ZZ_2^n \times \ZZ_4^m,\ Q_8\times \mathbb{Z}_2^n,\ S_3,\ {\rm and}~{\rm Dic}_{12},$$ 
where $m,n$ are arbitrary non-negative integers, $Q_8$ is the quaternion group of order $8$, and\/ ${\rm Dic}_{12}$ is the dicyclic group of order $12$.
\end{thm}

We note that the dicyclic group of order $12$ can be described as the non-abelian semi-direct product $\mathbb{Z}_3\rtimes \mathbb{Z}_4$.  One of the interesting features is that it has $S_3$ as a homomorphic image.   We also point out that $S_3$ and ${\rm Dic}_{12}$ are the only non-nilpotent groups on the list.

Let us first describe some basic properties of Cayley integral groups.

\begin{lem}
\label{subgroup integral}
Let $G$ be a Cayley integral group.  Then every subgroup and every homomorphic image of $G$ is also Cayley integral.
\end{lem}

\begin{proof}
The claim for subgroups is obvious since for a subset $S$ of a subgroup $H\le G$, The Cayley graph $\Cayley(G,S)$ consists of $[G:H]$ copies of $\Cayley(H,S)$.

Next, suppose that $K$ is a homomorphic image of $G$.  Let $\pi:G\to K$ be a surjective homomorphism and let $S$ be a symmetric subset of $K$. Let $T=\pi^{-1}(S)$.  Then $T$ is a symmetric subset of $G$.  We let $A_G$ denote the adjacency matrix of $\Cayley(G,T)$. Then $A_G=A_H\otimes J_k$, where $A_H$ is the adjacency matrix of $\Cayley(H,S)$, $k=|G|/|K|$, and $J_k$ is the $k\times k$ matrix with every entry equal to one. If ${\bf w}$ is an eigenvector of $A_H$ corresponding to an eigenvalue $\lambda$ and if ${\bf j}$ is the $k\times 1$ matrix whose entries are all $1$, then $A_G({\bf w}\otimes {\bf j})=k\lambda ({\bf w}\otimes {\bf j})$.  Since $\lambda$ is an algebraic integer, it must indeed be an integer.  This implies that $K$ is a Cayley integral group.
\end{proof}

We now show that the property of being Cayley integral is equivalent to a weaker property.

\begin{prop}
A group $G$ is Cayley integral if and only if every connected Cayley graph of $G$ is integral.
\end{prop}

\begin{proof}
One direction is obvious.
Suppose now that every connected Cayley graph of $G$ is integral, but there is a subset $S$ of $G$ such that $\Cayley(G,S)$ is not integral. Let $T = G\setminus(S\cup\{1\})$. Note that $\Cayley(G,S)$ is disconnected, thus its complementary graph, which is equal to  $\Cayley(G,T)$ is connected. Thus, by the assumption, $\Cayley(G,T)$ is integral. It is well-known that the complement of a regular integral graph is also integral. This contradicts our assumption that $\Cayley(G,S)$ has non-integral eigenvalues.
\end{proof}

We note that the only Cayley integral cyclic groups are $\mathbb{Z}_m$ with $m\in \{1,2,3,4,6\}$.  Thus if $G$ is a Cayley integral group, then since subgroups of $G$ are also Cayley integral, $G$ can not have any elements of order $p$ where $p$ is a prime greater than 3.  In particular, Cauchy's theorem gives that $G$ is a $(2,3)$-group; i.e., the order of $G$ is the product of a power of $2$ and a power of $3$.  A theorem of Burnside then gives that $G$ is necessarily a solvable group.

We summarize the important points obtained so far in the following remark.
\begin{rem}
Let $G$ be a group.  Then:
\begin{enumerate}
\item $G$ is a Cayley integral group if and only if $\Cayley(G,S)$ is an integral graph for every symmetric generating sets $S$ of $G$;
\item if $G$ is a Cayley integral group then so are subgroups and homomorphic images of $G$;
\item if $G$ is a Cayley integral group then its order is a product of a power of $2$ and a power of $3$ and all elements of $G$ have order in $\{1,2,3,4,6\}$;
\item $G$ is a solvable group.
\end{enumerate}
\end{rem}
We will make use of this remark often without referring to it directly.

We next give a result that will be used to characterize Cayley integral groups.  It shows, roughly speaking, that if a group $G$ has a symmetric generating set $S$ such that $\Cayley(G,S)$ is an integral graph, then $|G|$ cannot be too large compared to $|S|$.

\begin{prop}
\label{prop: bound}
Let $G$ be a finite group and let $S$ be a symmetric generating set of $G$.
If $\Cayley(G,S)$ is an integral Cayley graph, then the order of $G$ divides $2(2|S|-1)!$.
If, in addition, $G$ is perfect or $S$ has an element of odd order, then $|G|$ divides $(2|S|-1)!$.
\end{prop}

\begin{proof}
Let $A_S$ denote the adjacency matrix of $\Cayley(G,S)$.  For each group element $g \in G$, we let $A_g$ denote the permutation matrix (associated with the left-regular representation of $G$) of $g$.  We then have that $A_S = \sum_{s \in S} A_s$.  Let $k=|S|$. Since $S$ is a symmetric generating subset of $G$, $\Cayley(G,S)$ is a $k$-regular connected graph. Therefore all eigenvalues of $A_S$ are in the set $\{-k, \ldots, k-1,k\}$.
A well-known consequence of the Perron-Frobenius Theorem is that the eigenspaces of the eigenvalues $k$ and $-k$ are at most $1$-dimensional since the graph is connected.  Moreover, $-k$ is an eigenvalue if and only if the graph is bipartite.
We now look at the cases corresponding to whether $\Cayley(G,S)$ is bipartite or not.

\medskip
\noindent \emph{Case I.} $\Cayley(G,S)$ is not bipartite.

\medskip

\noindent In this case $-k$ is not an eigenvalue of $A_S$. Since $A_S$ is a symmetric matrix, it is diagonalizable and therefore the minimal polynomial of $A_S$ divides
$$(x-k) \prod_{i=-k+1}^{k-1}(x-i).$$
If we take $\Phi (x)=\prod_{i=-k+1}^{k-1}(x-i)$, then $B:=\Phi(A_S)$ will be nonzero, since $A_S$ has $k$ as an eigenvalue.  Let ${\bf j}$ be the vector whose coordinates are all equal to one.  This spans the kernel of $A_S - kI$.
Since $B$ is nonzero, there is some $i$ such that $B{\bf e}_i$ is nonzero, where ${\bf e}_i$ is the vector with a one as its $i$-th coordinate and zeros in every other coordinate.  Moreover, $(A_S - kI)B=0$ and so $B{\bf e}_i=c\,{\bf j}$ for some $c\in \mathbb{Z}\setminus\{0\}$.
Then ${\bf j}^T B = {\bf j}^T \prod_{i=-k+1}^{k-1} (A_S-iI) = (2k-1)!\,{\bf j}^T$.  Hence,
$$(2k-1)!=(2k-1)!\,{\bf j}^T\cdot {\bf e}_i = {\bf j}^T B{\bf e}_i = c\,{\bf j}^T \cdot {\bf j} = c |G|.$$
It follows that $|G|$ divides $(2k-1)!$ in this case.  Notice that this case necessarily occurs if $S$ contains an element of odd order. It also occurs when $G$ is perfect. To see this, note that $\Cayley(G,S)$ being bipartite implies that there is a homomorphism $\phi$ from $G$ to $\ZZ_2$ which sends each element in $S$ (and all elements in the bipartite class containing $S$) to 1. The kernel of $\phi$ must contain $G'$ since the image is abelian, and so if $G$ is perfect then $\phi$ would need to be trivial.

\medskip
\noindent \emph{Case II.} $\Cayley(G,S)$ is bipartite.

\medskip
\noindent We let ${\bf u}$ be a nonzero integer vector with $A_S{\bf u}=-k{\bf u}$.  We can take ${\bf u}$ to be the vector whose coordinates are all in $\{\pm 1\}$, where we have a $1$ in the $g$-th coordinate if and only if $g$ is in the kernel of the homomorphism from $G$ to $\ZZ_2$ that sends each element of $S$ to $1$.

As before we let $B=\Phi(A_S)$, where $\Phi$ is the polynomial described in Case I.  Then
$(A-kI)(A+kI)B=0$ and so the range of $B$ is contained in the span of ${\bf j}$ and ${\bf u}$.
Moreover, $B$ is nonzero since $k$ and $-k$ occur as eigenvalues of $|A|$.  Thus there is some $i$ such that $B{\bf e}_i=c\,{\bf j} + d\,{\bf u}$ for some $c,d\in \mathbb{Q}$, not both zero, with $c\,{\bf j}+d\,{\bf u}$ a vector with integer coordinates.  Notice that this implies that $c+d$ and $c-d$ are integers.

Since $A_S$ is Hermitian and ${\bf u}$ and ${\bf j}$ are eigenvectors from distinct eigenspaces, we see that ${\bf u}$ and ${\bf j}$ are orthogonal.  As before, we have
$${\bf j}^T B = \Phi(k)\,{\bf j}^T = (2k-1)!\,{\bf j}^T~\qquad ~{\rm ~and}\qquad {\bf u}^T B = \Phi(-k){\bf u}^T = -(2k-1)!\,{\bf u}.$$
Thus
\begin{equation}
\label{eq:cd1}
  (2k-1)!=(2k-1)!\,{\bf j}^T\cdot {\bf e}_i={\bf j}^T B{\bf e}_i = c\,{\bf j}^T\cdot {\bf j} = c|G|
\end{equation}
and 
\begin{equation}
\label{eq:cd2}
-(2k-1)!=-(2k-1)!{\bf u}^T \cdot {\bf e}_i = d{\bf u}^T\cdot {\bf u} = d|G|.
\end{equation}
By summing up (\ref{eq:cd1}) and (\ref{eq:cd2}), we see that $(c+d)|G|=0$, thus $d=-c$. By taking the difference, we obtain $2c|G| = 2(2k-1)!$. Since $2c=c-d$ is an integer, we conclude that $|G|$ divides $2(2k-1)!$. 
\end{proof}

We now classify all Cayley integral groups.  During the course of giving our classification, it will be useful to understand whether some groups of small order are Cayley integral or not.

\begin{lem}
\label{lem: list0}
The following groups are Cayley integral groups:
\begin{enumerate}
\item[\rm (a)] $S_3$,
\item[\rm (b)] the dicyclic group ${\rm Dic}_{12}$ (the non-trivial semi-direct  product $\mathbb{Z}_3\rtimes \mathbb{Z}_4$),
\item[\rm (c)] $Q_8\times \mathbb{Z}_2^d$ for every $d\ge 0$.
\end{enumerate}
\end{lem}

\begin{proof}
Notice that (a) follows from (b) since ${\rm Dic}_{12}$ has $S_3$ as homomorphic image.  We note that ${\rm Dic}_{12}$ has $\langle x,y~|~x^3=y^4=1,~yxy^{-1}=x^{-1}\rangle$ as a presentation.
Any symmetric subset $S$ of ${\rm Dic}_{12}$ is a union of sets from
$\{1\}$, $\{x,x^2\}$, $\{y,y^3\}$, $\{y^2\}$, $\{xy,xy^3\}$, $\{x^2y,x^2y^3\}$, and $\{xy^2,x^2y^2\}$.
Moreover, $y^2$ is central and hence gets mapped to either the identity or to the negative of the identity by any irreducible representation.  We consider these cases separately.
If $y^2$ is sent to $-I$ then each of $y+y^3$, $xy+xy^3$, and $x^2y+x^2y^3$ is sent to zero; and each of $xy^2+x^2y^2$, $y^2$, $x+x^2$, and $1$ is sent to an integer scalar matrix.  Thus each symmetric set $S$ has the property that the corresponding element of the group algebra is sent to an integer scalar multiple of the identity and hence has integer eigenvalues.  If, on the other hand, $y^2$ is sent to $I$ then our representation factors through ${\rm Dic}_{12}/\langle y^2\rangle \cong S_3$.  Notice that if we let $\pi$ denote the isomorphism from ${\rm Dic}_{12}/\langle y^2\rangle $ to $S_3$, in which the image of $x$ is sent to $(123)$ and the image of $y$ is sent to $(12)$, then we see that the symmetric set $S$ becomes a multi-set in which we have at most two copies of $\{{\rm id}\}$, at most three copies of $\{(123),(132)\}$, and either zero or two copies of each of $\{(12)\}$, $\{(13)\}$, and $\{(23)\}$.

Both ${\rm id}$ and $(123)+(132)$ are central in the group algebra and since the characters of $S_3$ are integer-valued we see that these elements are sent to integer multiples of the identity in any irreducible representation of $S_3$.  Thus these sets have no affect on whether we obtain a matrix with integer eigenvalues.  Thus we may assume that the multi-set is a union consisting of either $0$ or $2$ copies of each of $\{(12)\}$, $\{(13)\}$, $\{(23)\}$.  Notice that these elements each have order $2$ and so if the multi-set has size $2$ (i.e., we have two copies of a single transposition) then we obtain a matrix with eigenvalues in $\{\pm 2\}$. Next, observe that $(12)+(13)+(23)$ is central and since the characters of $S_3$ are integer-valued, we see that if $S$ has size $6$ then we again obtain a matrix with integer eigenvalues.  Finally, if our multi-set has size $4$ then by applying an inner automorphism we may assume that it is given by $\{(12),(12),(13),(13)\}$.  Then $(12)+(13)$ maps to $2$ under the trivial representation; to $-2$ under the alternating representation; and has the same image as $-(23)$ in the irreducible $2$-dimensional representation of $S_3$.  Thus we see that in each case we obtain a matrix with integer eigenvalues.  This establishes (a) and (b).

To show (c), let $G=Q_8 \times \mathbb{Z}_2^d$ and let $z$ be the central element of order $2$ in $Q_8$. If $\phi$ is an irreducible representation of $G$ then $z$ must either be sent to the identity or to the negative of the identity.  If $z$ is sent to the identity then $\phi$ in fact factors through $G/\langle z\rangle$, which is an elementary abelian $2$-group and thus $\phi$ is one-dimensional and clearly any symmetric set will be sent to an integer.  If, on the other hand, $\phi(z)=-I$, then notice that if $u$ is an element of order $4$ then $u^2=z$ and so the natural extension of $\phi$ to the group algebra of $G$ sends $u+u^{-1}=u(1+z)$ to $0$.  Consequently, we only need to consider symmetric sets consisting of elements of order $2$.  But all elements of order $2$ are central in $G$ and hence are mapped to either $I$ or $-I$ by $\phi$.  It follows that ${\rm Cay}(G,S)$ is an integer Cayley graph for each symmetric subset $S$ of $G$, giving (c).
\end{proof}

\begin{lem}
\label{lem: list}
The following groups are not Cayley integral groups:
\begin{enumerate}
\item[\rm (1)] any dihedral group $D_n$ with $n\ge 4$;
\item[\rm (2)] any non-abelian group of order $12$ that is not isomorphic to ${\rm Dic}_{12}$;
\item[\rm (3)] any non-abelian group of order $18$;
\item[\rm (4)] any non-abelian group of order $24$;
\item[\rm (5)] $Q_8\times \mathbb{Z}_4$.
\end{enumerate}
\end{lem}

\begin{proof}
We first show (1).  If $n\ge 4$ and $n\not \in \{4,6\}$ then $D_n$ contains an element that is of order $r\notin \{1,2,3,4,6\}$ and thus $D_n$ is not Cayley integral (since the subgroup isomorphic to $\ZZ_r$ is not).  Thus we only need to worry about $n\in \{4,6\}$.   Notice that $D_n$ has the presentation $\langle x,y~|~x^2=y^n=1, xyx=y^{-1}\rangle$.  We have a $2$-dimensional representation $\theta$ of $D_n$ given by
\[ x \mapsto \left( \begin{array}{cc} 0 & 1  \\ 1 & 0  \end{array}\right),~~~y \mapsto \left( \begin{array}{cc} \omega_n & 0 \\ 0 & \omega_n^{-1} \end{array}\right),\] where $\omega_n$ is the primitive $n$-th root of unity.
Then if we use the symmetric generating set $S=\{x,xy\}$ we see that $\theta(x)+\theta(xy)$ is given by
\[  \left( \begin{array}{cc}0& 1+\omega_n  \\ 1+\omega_n^{-1} & 0  \end{array}\right),\] which has eigenvalues
$\pm \sqrt{2+\omega_n+\omega_n^{-1}}$.  We note that if $n=4$ then this gives eigenvalues $\pm\sqrt{2}$ and if $n=6$ this gives eigenvalues $\pm \sqrt{3}$.  Thus we have (1).

We now consider (2).   Notice that the only non-abelian groups of order $12$ are, up to isomorphism, ${\rm Dic}_{12}$, $A_4$, and $D_6$.  By (1), we only need to consider $A_4$.  For $A_4$ notice that if we use the $4$-dimensional representation $\rho$ which associates to a permutation in $A_4$ its corresponding permutation matrix and if we use the symmetric set
$S=\{(13)(24),(14)(23),(123),(132)\}$, then by extending $\rho$ to the group algebra of $A_4$ via linearity, we see that
$(13)(24)+(14)(23)+(123)+(132)$ is represented by the matrix
\[  \left( \begin{array}{rrrr} 0 & 1 & 2 & 1 \\ 1 &0 & 2 & 1\\ 2 & 2 & 0 & 0 \\  1 & 1 &0  &  2 \end{array}\right),\] which has eigenvalues
$4, -1, \frac{-1\pm \sqrt{17}}{2}$.  Thus $A_4$ is not Cayley integral.

To prove (3), we note that up to isomorphism there are only three non-abelian groups of order $18$: $D_9$, $S_3\times \mathbb{Z}_3$, and the group $E_9\cong \mathbb{Z}_3^2\rtimes_{\theta} \mathbb{Z}_2$, where $\theta$ is the map that sends every element of $\mathbb{Z}_3^2$ to its inverse.  The group $D_9$ is not Cayley integral by (1).  For $S_3\times \langle x~|~x^3=1\rangle$, we take the representation that sends
$(\sigma, x^j)\mapsto \omega^j P(\sigma),$ where $\omega$ is the primitive third-root of unity and $P$ is the (reducible) $3$-dimensional representation of $S_3$ that associates to $\sigma\in S_3$ the $3\times 3$ permutation matrix $P(\sigma)$ of $\sigma$.  If we extend this to the group algebra via linearity, then the symmetric element
$((12),x)+((12),x^2)+((13),1)$ is represented by the matrix
\[  \left( \begin{array}{rrr} 0 & -1 & 1 \\ -1 &1 & 0 \\ 1 & 0 & -1 \end{array}\right),\]
which has eigenvalues $0$, $\pm \sqrt{3}$.  Thus $S_3\times \mathbb{Z}_3$ is not Cayley integral.
The group $E_9$ has presentation $\langle x,y~|~x^3=y^3=[x,y]=1\rangle \rtimes \langle z ~|~z^2=1\rangle$, where the automorphism of $\langle x,y\rangle$ determining the semidirect product is $x\mapsto x^{-1}$, $y\mapsto y^{-1}$.  Notice that $xz,yz$, and $z$ all have order $2$.  Thus we may consider the symmetric set $S=\{xz,z,yz\}$.  We claim that the element $xz+z+yz$ in the group algebra has some representation with eigenvalues that are not all integers.  To see this, observe that $\langle x\rangle$ is a normal subgroup of $E_9$ and when we mod out by this group we have a group isomorphic to $S_3$ with isomorphism given by $\bar{y}\mapsto (123)$, $\bar{z}\mapsto (12)$.  Then the image of $xz+z+yz$ in the group algebra of $S_3$ under the composition of maps described above is $2(12)+(13)$.  Notice that the $3$-dimensional permutation representation of $S_3$ sends this element to
\[\left( \begin{array}{ccc} 0 & 2 & 1 \\ 2 & 1 & 0 \\ 1 & 0 & 2 \end{array}\right),\] which has eigenvalues $\{3,\pm \sqrt{3}\}$.  Notice that this representation lifts to a representation of $E_9$ and thus we see that $E_9$ is not Cayley integral.

To prove (4), we note that up to isomorphism there are $15$ groups of order $24$,  $3$ of which are abelian.  Of the remaining $12$ there are only two that do not have any elements of order $8$ or $12$, do not contain a copy of $D_4$, and do not contain a copy of a non-abelian group of order $12$ that is not isomorphic to ${\rm Dic}_{12}$. (These are necessary properties to be Cayley integral by (1) and (2).)  These two groups are ${\rm SL}_2(\mathbb{Z}_3)$ and ${\rm Dic}_{12}\times \mathbb{Z}_2$, up to isomorphism.  Notice that $S_3$ is a homomorphic image of ${\rm Dic}_{12}$ and so $S_3\times \mathbb{Z}_2$ is a homomorphic image of ${\rm Dic}_{12}\times \mathbb{Z}_2$.  But $S_3\times \mathbb{Z}_2$ is not Cayley integral by (2) and so neither is ${\rm Dic}_{12}\times \mathbb{Z}_2$. Note that $A_4$ is isomorphic to ${\rm PSL}_2(\mathbb{Z}_3)$ and hence $A_4$ is a homomorphic image of ${\rm PSL}_2(\mathbb{Z}_3)$.  This shows that ${\rm PSL}_2(\mathbb{Z}_3)$ is not Cayley integral by (2).  This establishes $(4)$.

Finally, to prove (5), we note that $Q_8=\langle x,y,z~|~x^2=y^2=[x,y]=z,~z^2=1\rangle$ has a representation $\pi$ determined by
\[ x \mapsto \left( \begin{array}{rr} i & 0  \\ 0 & -i  \end{array}\right),~~~y \mapsto \left( \begin{array}{rr} 0 & 1 \\ -1 & 0   \end{array}\right).\]  
Thus $Q_8\times \langle t~|~t^4=1\rangle$ has a $2$-dimensional representation $\rho$ given by
$\rho((a,t^j))=i^j \pi(a)$ for $a\in Q_8$ and $j\in \mathbb{Z}$, where $i$ is a primitive fourth root of unity.
If we use the symmetric set $S=\{(x,t),(x^{-1},t^{-1}),(y,t),(y^{-1},t^{-1})\}$, we see that $\rho$ sends the element
$(x,t)+(x^{-1},t^{-1})+(y,t)+(y^{-1},t^{-1})$ from the group algebra to
\[  \left( \begin{array}{rr} -2 & 2i \\ -2i & 2  \end{array}\right),\]
which has eigenvalues $\pm 2\sqrt{2}$ and hence $Q_8\times \mathbb{Z}_4$ is not Cayley integral, giving us (5).
\end{proof}

\begin{cor} 
\label{cor: sym}
Let $H\le S_4$ be a Cayley integral subgroup of $S_4$ that acts transitively on $\{1,2,3,4\}$.  Then $H$ has order $4$.
\end{cor}

\begin{proof}
We note that if $H$ has order in $\{8,12,24\}$ then $H$ is isomorphic to one of $D_4$, $A_4$, or $S_4$ and hence is not Cayley integral by Lemma \ref{lem: list} (1), (2), and (4).  Each subgroup of order $6$ is equal to the set of permutations that fix some element $i\in \{1,2,3,4\}$ and hence does not act transitively on $\{1,2,3,4\}$.  Thus $H$ has order in $\{1,2,3,4\}$.  It is straightforward to check that a subgroup of order $1,2$, or $3$ cannot act transitively on $\{1,2,3,4\}$ and thus $H$ has order $4$.
\end{proof}

We now start the classification of Cayley integral groups by first classifying the Cayley integral $2$-groups.

\begin{lem}
\label{lem: 2} 
Let $Q$ be a Cayley integral $2$-group.  Then the following statements hold:
\begin{enumerate}
\item[\rm (i)] Every element of order $2$ is central.
\item[\rm (ii)] If $Q$ is non-abelian then any two elements that do not commute generate a subgroup that is isomorphic to $Q_8$.
\end{enumerate}
\end{lem}

\begin{proof}
Let $N$ denote the set of elements in $Q$ of order at most $2$.  We claim that $N$ is a group.  To see this, it is sufficient to show that if $x,y\in N$ then $xy=yx$ since this implies that $(xy)^2 = x^2y^2 = 1$.  This will show that the set of elements of order at most $2$ is closed under multiplication and hence forms a group.  Moreover, it follows that $N$ is abelian.  Let $x,y\in N$ and let $E$ denote the subgroup of $N$ generated by $x$ and $y$. Then $E$ is a Cayley integral group and applying Proposition \ref{prop: bound} to the symmetric set $S=\{x,y\}$, we see that $|E|$ divides $12$.  Since $E$ is in $Q$ and $Q$ is a $2$-group, we see that $E$ has order at most $4$ and thus is abelian.  This means that $x$ and $y$ commute and since they were arbitrary elements of $N$, we thus have that $N$ is an abelian group as claimed.  We note that $N$ is normal, since the set of elements of order at most $2$ is closed under conjugation.

To complete the proof of (i), we must show that $N$ is a central subgroup of $Q$.  Suppose, towards a contradiction, that $N$ is not central.  Then there is some $u\in Q$ such that conjugation by $u$ induces a non-trivial automorphism of $N$.  Note that every element of $Q$ has order dividing $4$. Therefore $u^2\in N$ and so this automorphism must have order $2$.  Hence there are $x,y\in N$ with $x\neq y$ such that $uxu^{-1}=y$ and $uyu^{-1}=x$.  Let $Q_1$ denote the subgroup of $Q$ generated by $x,y$, and $u$.  Then $Q_1$ in non-abelian and has order $8$ or $16$.  Notice that $Q_1$ has at least four elements of order $4$ and hence must be isomorphic to $D_4$ if it has order $8$; but $D_4$ is not Cayley integral by Lemma \ref{lem: list} and so we see $Q_1$ must have order $16$.  In particular, $u$ has order $4$ and $\langle u\rangle$ intersects $\langle x,y\rangle $ trivially.  Thus $Q_1/\langle u^2\rangle$ is a Cayley integral group of order $8$ and, as before, we see that it is isomorphic to $D_4$,
a contradiction.  It follows that each element of $N$ is indeed central, which establishes (i).

We now prove (ii).  Suppose that $x,y \in Q$ and that they do not commute.  By (i), $x$ and $y$ must both have order at least $4$. But since $\langle x\rangle$ and $\langle y\rangle$ are Cayley integral, their order is equal to 4. Notice that since the square of every element is central, $Q/N$ is elementary abelian and so $Q'\subseteq N$.  In particular, $[x,y]=z$, where $z\in Q$ is a central element of order $2$.  We claim that $x^2=y^2=z$.  To see this, suppose that $x^2\neq z$ and let $H$ denote the subgroup of $Q$ generated by $x$ and $y$.  Then $E:=H/\langle x^2\rangle$ is a Cayley integral $2$-group and the image of $x$ in $E$ now has order $2$ and so it must be central.  But the image of $[x,y]=z$ in $E$ is non-trivial, a contradiction since by (i) we have that every element of order $2$ in $E$ is central.  It follows that $x^2=y^2=z$ and so $H$ is a non-abelian homomorphic image of the group with presentation
$$\langle s,t,u~|~s^4=t^4=u^2=1, s^2=t^2=[s,t]=u, [s,u]=[t,u]=1\rangle.$$
We note that this is just a presentation of $Q_8$ and since $H$ is non-abelian, we see that $H\cong Q_8$.
\end{proof}

\begin{prop} 
Let $Q$ be a non-abelian Cayley integral $2$-group.  Then $Q\cong Q_8\times \mathbb{Z}_2^d$ for some $d\ge 0$.
\label{prop: 2class}
\end{prop}

\begin{proof}
By Lemma \ref{lem: 2}, every element of order $2$ in $Q$ is central and any pair of non-commuting elements of $Q$ generate a subgroup that is isomorphic to $Q_8$. Moreover, every element is of order 1, 2, or 4.
Let $u,v$ be elements of order $4$ that generate a copy of $Q_8$.  Then there is a central element $z$ of order $2$ such that
$u^2=v^2=[u,v]=z$.  We claim that if $w$ is another element of order $4$ then $w^2 =z$.  To see this, note that if $w^2\neq z$ then $w$ and $u$ must commute since otherwise by Lemma \ref{lem: 2} (ii) they generate a copy of $Q_8$ with $w^2=[u,w]=u^2=z$.  Similarly, $[w,v]=1$ and since $w^2$ is central and not in $\{1,u^2\}$, we see that the group generated by $u,v$, and $w$ is isomorphic to $Q_8\times \mathbb{Z}_4$, which is not Cayley integral by Lemma \ref{lem: list} (5).  It follows that all elements of order $4$ in $Q$ have the same square.

Let $Z$ denote the central subgroup of $Q$ consisting of elements of order at most $2$.  By assumption, there exist $u$ and $v$ that do not commute and hence there is some $z\in Z$ such that $u^2=v^2=[u,v]=z$.  We claim that $Q$ is generated by $u$, $v$, and $Z$.  To see this, let $Q_0$ denote the subgroup of $Q$ generated by $u,v$, and $Z$ and suppose that there is some $w\in Q\setminus Q_0$.  Then $w$ has order $4$ and so $w^2=z$.  If $u$ and $w$ commute then $(uw)^2=u^2 w^2=z^2=1$ and so $uw\in Z$, which gives that $w\in Q_0$, a contradiction.  Thus $u$ and $w$ do not commute, which gives that $u^2=w^2=[u,w]=z$ by Lemma \ref{lem: 2} (ii).  Similarly, we have $v^2=w^2=[v,w]=z$.  Notice that $(uvw)^2 = 1$ and so $uvw\in Z$, which gives that $w\in v^{-1}u^{-1}Z\subseteq Q_0$, a contradiction.  Thus $Q=Q_0$ and so $Q$ is generated by $u,v$ and $Z$.  Now let $H$ be the subgroup of $Q$ generated by $u$ and $v$.  Then $H\cong Q_8$ and $H\cap Z=\langle z\rangle$.  Note that $Z$ is an elementary abelian $2$-group and so there is an elementary abelian subgroup $Z_1$ such that $Z_1\oplus \langle z\rangle = Z$.  Then we see that $Q\cong H\times Z_1\cong Q_8\times \mathbb{Z}_2^d$ for some $d\ge 0$.
\end{proof}

We now classify Cayley integral $3$-groups.  As it turns out, the classification in this case is simpler.

\begin{prop} 
Every Cayley integral $3$-group is elementary abelian.
\label{lem: 3group}
\end{prop}

\begin{proof}
Let $x$ and $y$ be two elements of a Cayley integral group $P$ and let $P_0$ denote the subgroup of $P$ generated by $x$ and $y$.  Then $P_0$ is Cayley integral and so applying Proposition \ref{prop: bound} to the symmetric set $S=\{x,x^{-1},y,y^{-1}\}$ gives that the order of $P_0$ divides $7!$.  Since $P_0$ is a $3$-group, we see that $|P_0|$ divides $9$.  In particular $P_0$ is abelian and so $x$ and $y$ commute.  Since all elements of $P$ commute, we see that $P$ is abelian.  Since every element of $P$ has order $1$ or $3$, we see that $P$ is an elementary abelian $3$-group.
\end{proof}

\begin{cor} 
Let $G$ be a nilpotent Cayley integral group.  Then $G$ is either abelian or $G\cong Q_8\times \mathbb{Z}_2^d$ for some $d\ge 0$.
\label{cor: xxxx}
\end{cor}

\begin{proof} 
If $G$ is nilpotent then $G$ must be a direct product of a Cayley integral $2$-group and a Cayley integral $3$-group.  Thus by Propositions \ref{prop: 2class} and \ref{lem: 3group}, if $G$ is non-abelian then
$G\cong (Q_8\times \mathbb{Z}_2^d)\times \mathbb{Z}_3^e$ for some $d,e\ge 0$.  Note that if $e\ge 1$ then $G$ contains a copy of $Q_8\times \mathbb{Z}_3$ which is not Cayley integral by Lemma \ref{lem: list} (4).  Hence $e=0$ and the result follows.
\end{proof}

We now begin to study non-nilpotent Cayley integral groups.  We first show that such groups necessarily have a unique Sylow $3$-subgroup.  To do this, we first require a few lemmas.

\begin{lem} 
Let $G$ be a Cayley integral group.  If $G$ has a normal Sylow $2$-subgroup then $G$ is nilpotent.
\label{lem: nilp}
\end{lem}

\begin{proof}
Suppose that this is not the case.  Then we can pick a non-nilpotent Cayley integral group $G$ of smallest order with respect to having a normal Sylow $2$-subgroup.

Let $Q$ denote the Sylow $2$-subgroup of $G$ and let $Z$ denote the center of $Q$.  Let $P$ be a Sylow $3$-subgroup of $G$.  Then $G$ is a semi-direct product $P\rtimes Q$.  Since $Z$ is a characteristic subgroup of $Q$ and $Q$ is normal in $G$, we see that if $x\in P$ then $xZx^{-1}=Z$.  Pick $z\in Z$ of order $2$.  We claim that $z$ commutes with every element of $P$.  To see this, suppose towards a contradiction, that there is some $x\in P$ such that $xz\neq zx$. Then $z_1:=xzx^{-1}$ and $z_2:=x^2 zx^{-2}$ have the property that the subgroup of $Z$ generated by $z,z_1,z_2$ is an elementary abelian $2$-group of order either $4$ or $8$ and hence the group generated by $x$ and $z$ must have order $12$ or $24$.  By Lemma \ref{lem: list}, the only non-abelian Cayley integral group of order either $12$ or $24$ is isomorphic to the dicyclic group of order $12$, but this one does not have a normal Sylow $2$-subgroup, a contradiction.

Thus we see that $xz=zx$ for every $z\in Z$ and $x\in P$. This means that the centralizer of $z$ contains both $P$ and $Q$ and thus must contain all of $G$.  Notice that $H:=G/\langle z\rangle$ is a Cayley integral group with the property that it has a normal Sylow $2$-subgroup.  By minimality of the order of $G$ we see that $H$ is nilpotent.  It follows that $G$ is nilpotent, since we obtained $H$ by taking the quotient of $G$ with a central subgroup.
\end{proof}

\begin{lem}  
Let $G$ be a Cayley integral group generated by two elements of order $3$. Then $G$ is isomorphic to $\ZZ_3$ or to $\ZZ_3 \times \ZZ_3$.
\label{lem: l144}\label{lem: 144}
\end{lem}

\begin{proof}
Let $x$ and $y$ be elements of order $3$ in $G$ that generate $G$ as a group.
Notice that the set $S=\{x,x^{-1},y,y^{-1}\}$ has size $4$ and since $x$ has odd order we see from Proposition \ref{prop: bound} that the order of $G$ divides $7!$. Since $G$ is a $(2,3)$-group, we see that the order of $G$ in fact divides $144$.  

We let $n_3$ denote the number of Sylow $3$-subgroups of $H$.  It is known that $n_3\equiv 1$ (mod $3$) and that $n_3$ divides the index of the Sylow subgroup in $G$. Since $|G|$ divides 144, the index is $2^t$, where $0\le t\le 4$, thus $n_3\in \{1,4,16\}$. 

If $n_3=1$, then $G$ has a unique Sylow $3$-subgroup, which is elementary abelian by \ref{lem: 3group}.  Hence $x$ and $y$ commute and so they generate a group of order $3$ or $9$.  This yields the conclusion of the lemma.

In the rest of the proof we argue by contradiction, considering the cases $n_3=4$ and $n_3=16$ separately.

Suppose that $n_3=4$.  Then $G$ acts on the Sylow $3$-subgroups by conjugation, which gives us a non-trivial homomorphism $\pi$ from $G$ to $S_4$.   Let $G_0$ denote the image of $G$ under $\pi$.  Since the collection of Cayley integral groups is closed under the process of taking subgroups and homomorphic images, $G_0$ is a Cayley integral subgroup of $S_4$. Moreover, by construction $G_0$ acts transitively on $\{1,2,3,4\}$ since $G$ acts transitively on the set of Sylow $3$-subgroups under conjugation.  By Corollary \ref{cor: sym}, $G_0$ has order $4$.  Let $N$ denote the kernel of $\pi$.  Then $N$ has order dividing $36$ and by construction it contains all Sylow $3$-subgroups and in particular contains $x$ and $y$.  But this means that the group generated by $x,y$ is contained in $N$, a contradiction since $N$ is a proper subgroup of $G$.  We conclude that $n_3=4$ cannot occur.

Suppose next that $n_3=16$.  
Suppose first that $|G|\ne 144$. Since $n_3=16$, we know that $16$ divides the order of $G$ and since $G$ is a proper divisor of $144$ and $3$ divides the order of $G$, we see that $|G|=48$.  Then each pair of distinct Sylow $3$-subgroups must intersect trivially since they are all cyclic groups of order $3$.  Thus there are $16\cdot 2=32$ elements of order $3$.  This leaves $16$ unaccounted elements, which necessarily make up a normal Sylow $2$-subgroup.  By Lemma \ref{lem: nilp}, we see that $G$ is nilpotent and thus $n_3=1$, a contradiction.

Suppose now that $|G|=144$.
If each pair of distinct Sylow $3$-subgroups intersect trivially then $G$ has $8\cdot n_3=128$ elements of order $3$.  This leaves $16$ unaccounted for elements in $G$, which must make up a normal Sylow $2$-subgroup.  By Lemma \ref{lem: nilp}, $G$ is nilpotent, which gives that $n_3=1$, a contradiction.

Thus $G$ has distinct Sylow $3$-subgroups $P$ and $Q$ such that $P\cap Q=\langle u\rangle$ is a group of order $3$.  Notice that $P$ and $Q$ both have order $9$ and hence are abelian.  It follows that $C_G(u)$, the centralizer of $u$ in $G$, contains the groups $P$ and $Q$. It follows that its order is a multiple of $9$ and since it contains two distinct Sylow $3$-subgroups it must have at least four Sylow subgroups and so its order must in fact be in $\{36,72,144\}$.

Our next step is to show that $C_G(u)$ is normal in $G$.  If $C_G(u)$ has order $72$ or $144$, this is automatic, so we may assume that $|C_G(u)|=36$.  Then $G$ acts on the left cosets of $C_G(u)$, giving a homomorphism $\rho$ to $S_4$.  Let $E$ denote the image of $\rho$ in $S_4$.  By assumption the image of $\rho$ is a Cayley integral group that acts transitively on $\{1,2,3,4\}$ and hence $E$ must have order $4$ by Corollary \ref{cor: sym}.  Thus the kernel of $\rho$ has size $36$ and since the kernel of $\rho$ is contained in $C_G(u)$, we see that $C_G(u)$ is normal in this case.

We now show that $C_G(u)=G$.  Since $G$ is generated by $x$ and $y$, it is sufficient to show that $u$ commutes with $y$.  Let $Z_1$ denote the Sylow $3$-subgroup of the center of $C_G(u)$.  Note that $Z_1$ is characteristic in $C_G(u)$ and hence normal in $G$. Moreover, $Z_1$ is non-trivial since $u\in Z_1$.  Notice that if $Z_1$ has order $9$ then it is a Sylow subgroup of $G$ and since all Sylow subgroups are conjugate and $C_G(u)$ is normal we see that $x$ and $y$ are in $C_G(u)$, which gives that $G=C_G(u)$ since $x$ and $y$ are generators of $G$.  Thus $Z_1=\langle u\rangle$.  Notice that $xZ_1x^{-1}=Z_1$ and so $xux^{-1}\in \{u,u^{-1}\}$.  If $xu=u^{-1}x$ then $u=x^3 u=u^{-1}x^3 =u^{-1}$, a contradiction.  Thus $xu=ux$.  Similarly, $yu=uy$, which gives that $x,y\in C_G(u)$ and so $C_G(u)=G$.

Now $H:=G/\langle u\rangle = C_G(u)/\langle u\rangle$ is a Cayley integral group of order $48$ and is generated by two elements of order $3$. Since we already proved the lemma for groups whose order is less than 144, we can apply the lemma to the group $H$. It follows that $H$ has order $3$ or $9$. This gives a contradiction and completes the proof by showing that $n_3\ne16$ when $|G|=144$.
\end{proof}

\begin{prop}
Let $G$ be a Cayley integral group.  Then $G$ has a normal abelian Sylow $3$-subgroup.
\label{thm: 3grp}
\end{prop}

\begin{proof}
By Lemma \ref{lem: 144} any two elements of order $3$ generate a group of order $3$ or $9$.  Since groups of orders $3$ and $9$ are abelian, it follows that any two elements of order $3$ commute.  Thus the product of two elements of order $3$ has order $1$ or $3$.  This shows that elements of order dividing $3$ are closed under multiplication in $G$ and hence form a group.  This group is necessarily the unique Sylow $3$-subgroup of $G$ and so $G$ has a normal Sylow $3$-subgroup.  By Proposition \ref{lem: 3group}, this group must be abelian.
\end{proof}

\begin{cor} 
Let $G$ be a non-nilpotent Cayley integral group.  Then $G$ is isomorphic to either $S_3$ or ${\rm Dic}_{12}$.
\label{nonnilp}
\end{cor}

\begin{proof}
By Proposition \ref{thm: 3grp}, $G$ has a normal Sylow $3$-subgroup, $P\cong \mathbb{Z}_3^d$. Moreover, $d\ge 1$ since $G$ is not nilpotent.  Let $Q$ be a Sylow $2$-subgroup of $G$. Then $G=P \rtimes Q$.

We first claim that if $x\in P$ and if $y\in Q$ has order $2$, then $yxy^{-1}\in \{x,x^{-1}\}$.  To see this, suppose that $yxy^{-1}=u\not \in \{x,x^{-1}\}$.  Then $u$ is of order 3 and $x$ and $u$ generate a group of order $9$ by Lemma \ref{lem: 144}. Consequently, $x,y,u$ generate a non-abelian subgroup of $G$ of order $18$.  But this contradicts Lemma \ref{lem: list} (3), since a non-abelian group of order $18$ cannot be Cayley integral.

We next claim that if $|P|\ge 9$ and if $y\in Q$ has order $2$, then $yx=xy$ for every $x\in P$.  To see this suppose that there is some $x\in P$ such that $yx\neq xy$.  As shown above, we have $yxy^{-1}=x^{-1}$.  Let $u\in P$ be such that $\langle x,u\rangle$ has order $9$.  Then since $yuy^{-1}\in \{u,u^{-1}\}$, we see that $u,x,y$ generate a non-abelian group of order $18$.  But this is a contradiction, since Lemma \ref{lem: list} says that no such group can be Cayley integral.  Thus we have shown that either $|P|=3$ or we have $yx=xy$ whenever $y\in Q$ has order $2$ and $x\in P$.

We next claim that if $|P|\ge 9$ and $w\in Q$ then $wxw^{-1}\in \{x,x^{-1}\}$ for every $x\in P$.  To see this, suppose that this is not the case.  Then $wxw^{-1}=u\not \in \{x,x^{-1}\}$. By the above, the order of $w$ is greater than 2 and since $G$ is Cayley integral and $w\in Q$, its order must be 4. Thus $w^2$ has order 2 and hence $w^2 x = xw^2$. This implies that $wuw^{-1}=x$ and so the group generated by $u,x,w$ is a non-abelian group of order $36$ and $w^2$ is central.  Notice that the quotient of the group generated by $u,x,w$ by $\langle w^2\rangle$ is a non-abelian group of order $18$ and hence it cannot be Cayley integral by Lemma \ref{lem: list} (3).  This is a contradiction and so we conclude that if $|P|\ge 9$ then whenever $x\in P$ we have that $\langle x\rangle$ is normal in $G$ since its normalizer contains both $P$ and $Q$.

We now claim that $|P|\le 3$. If $|P|\ge 9$, then notice that $P$ cannot be central in $G$ since $G$ is not nilpotent.  Thus there is some $y\in Q$ and some $x\in P$ such that $xy\neq yx$.  We have just shown that we must have $yxy^{-1}=x^{-1}$.  Pick $u\in P$ such that $u$ and $x$ generate a subgroup of $P$ of order $9$.  Then $\langle u,x\rangle$ is normal in $G$, the group $E$ generated by $y,u,x$ has order $36$, and $y^2$ is central in $E$.  But by construction, $E/\langle y^2\rangle$ is a non-abelian group of order $18$ and hence cannot be Cayley integral.  It follows that $|P|\le 3$, as claimed.  Moreover, since $G$ is not nilpotent, $|P|=3$.

We next claim that $Q$ is abelian.  If not, then $Q$ contains a copy of $Q_8$.  Then $G$ contains a copy of $P\rtimes Q_8$, which is not Cayley integral by Lemma \ref{lem: list} (4), since $P\rtimes Q_8$ is a non-abelian group of order $24$.  Thus $Q$ is abelian.

Finally, we claim that $Q$ has order at most $4$.  To see this, suppose that $|Q|\ge 8$.  By assumption, $G$ is non-nilpotent and so there is some $u\in Q$ such that conjugation by $u$ induces a non-trivial automorphism of $P$.  Since $Q$ is an abelian $2$-group of order at least $8$, there is a subgroup $Q_0$ of $Q$ of order $8$ that contains $u$.  Then $P\rtimes Q_0$ is a non-abelian group of order $24$ and so by Lemma \ref{lem: list} (4) is not Cayley integral, a contradiction.  Thus $Q$ has order at most $4$ and since $G$ is not nilpotent it must have order at least $2$.  Hence $G=P\rtimes Q$ has order $6$ or $12$.
Since $G$ is not nilpotent, we see by Lemma \ref{lem: list} that $G\cong S_3$ if $|G|=6$, and $G\cong {\rm Dic}_{12}$ if $|G|=12$.  This completes the proof.
\end{proof}

We are now ready to give the proof of the classification result for Cayley integral groups.

\begin{proof}[Proof of Theorem \ref{thm: main}.]
If $G$ is not nilpotent, then by Corollary \ref{nonnilp} we have that $G\cong {\rm Dic}_{12}$ or $G\cong S_3$.  If $G$ is nilpotent and non-abelian then by Corollary \ref{cor: xxxx} we see that $G\cong Q_8\times \mathbb{Z}_2^d$ for some $d\ge 0$.  If $G$ is abelian then by Theorem \ref{thm: KS} we have that $G\cong \mathbb{Z}_3^d \times \mathbb{Z}_2^e$ or $G\cong \mathbb{Z}_2^d\times \mathbb{Z}_4^e$ for some $d,e\ge 0$.  By Lemma \ref{lem: list0} all of these groups are Cayley integral.
\end{proof}

\end{document}